\documentclass[journal,twoside,web]{ieeecolor}
\usepackage{generic}
\usepackage{cite}

\usepackage{graphicx,algorithm,algpseudocode,amsthm,amsmath,amsfonts,verbatim,mathtools,enumerate,bm,hyperref,mathabx,hhline,multirow,tikz,braket}

\usepackage[style=base]{caption}
\usepackage{subcaption}

\usetikzlibrary{positioning,calc}
\usepackage[flushleft]{threeparttable}
\allowdisplaybreaks

\newcommand{\diag}{\mathop{\bf diag}}
\newcommand{\blkdiag}{\mathop{\bf blkdiag}}

\newcommand{\argmin}{\mathop{\rm argmin}}

\newcommand{\norm}[1]{\left\lVert#1\right\rVert}
\newcommand{\mnorm}[1]{{\left\vert\kern-0.25ex\left\vert\kern-0.25ex\left\vert #1 
    \right\vert\kern-0.25ex\right\vert\kern-0.25ex\right\vert}}

\newtheorem{theorem}{Theorem}
\newtheorem{lemma}{Lemma}

\newtheorem{remark}{Remark}

\newtheorem{assumption}{Assumption}

\newcommand{\eg}{{\it e.g.}}
\newcommand{\ie}{{\it i.e.}}

\usepackage{textcomp}
\def\BibTeX{{\rm B\kern-.05em{\sc i\kern-.025em b}\kern-.08em
    T\kern-.1667em\lower.7ex\hbox{E}\kern-.125emX}}

\begin{document}
\title{Proportional-Integral Projected Gradient Method for Model Predictive Control}
\author{Yue~Yu, Purnanand~Elango, and Beh\c{c}et~A\c{c}\i kme\c{s}e
\thanks{Accepted to IEEE Control System Letters, available at \url{https://ieeexplore.ieee.org/document/9295329}, DOI: 10.1109/LCSYS.2020.3044977}
\thanks{The authors are with the Department of Aeronautics and Astronautics, University of Washington, Seattle,
        WA 98195 USA (emails: 
         yueyu@uw.edu,pelango@uw.edu,behcet@uw.edu).}}

\maketitle
\thispagestyle{empty}
\pagestyle{empty}

\begin{abstract}
Recently there has been an increasing interest in primal-dual methods for model predictive control (MPC), which require minimizing the (augmented) Lagrangian at each iteration. We propose a novel first order primal-dual method, termed \emph{proportional-integral projected gradient method}, for MPC where the underlying finite horizon optimal control problem has both state and input constraints. Instead of minimizing the (augmented) Lagrangian, each iteration of our method only computes a single projection onto the state and input constraint set. Our method ensures that, along a sequence of averaged iterates, both the distance to optimum and the constraint violation converge to zero at a rate of \(O(1/k)\) if the objective function is convex, where \(k\) is the iteration number. If the objective function is strongly convex, this rate can be improved to \(O(1/k^2)\) for the distance to optimum and \(O(1/k^3)\) for the constraint violation. We compare our method against existing methods via a trajectory-planning example with convexified keep-out-zone constraints.
\end{abstract}


\section{Introduction}
\IEEEPARstart{M}{odel} predictive control (MPC) provides a systematic approach for automatic control with physical and operational constraints \cite{mayne2000constrained,mayne2014model,qin2003survey,kouro2008model,eren2017model}. The key ingredient in MPC is solving a finite horizon discrete-time convex optimal control problem that can be expressed in the following form 
\begin{equation}\label{opt: QP}
    \begin{array}{ll}
    \underset{z}{\mbox{minimize}} & \frac{1}{2}z^\top H z+h^\top z \\
    \mbox{subject to} & Gz=g, \enskip z\in \mathbb{Z},
    \end{array}
\end{equation}
where the trajectory variable \(z\) aims to minimize a convex quadratic cost function \(\frac{1}{2}z^\top H z+h^\top z\) subject to linear dynamics constraints \(Gz=g\) together with convex state and input constraint  \(z\in\mathbb{Z}\). 
Throughout we assume \(\mathbb{Z}\) is the Cartesian product of convex sets whose Euclidean projection can be evaluated at low computational cost. Such an assumption applies to many practical state and input constraints in MPC \cite{ullmann2011fiordos,richter2011computational,richter2013certification,jerez2014embedded}; see Tab.~\ref{tab: projection} and \cite{bauschke1996projection} for some examples. In addition to convex MPC problems, solution to problem \eqref{opt: QP} is also an integral part of problems with nonlinear dynamics and non-convex constraints. In these cases, a sequence of convex sub-problems modeled by \eqref{opt: QP} are solved to obtain the solution of the original non-convex problem, as done in sequential convex programming \cite{ghaemi2009integrated,bonalli2019gusto} and successive convexification methods \cite{mao2016successive,mao2017successive,mao2018successive}.  

\begin{table}[ht]
\begin{threeparttable}
\caption{Examples of simple convex sets and projections}\label{tab: projection}
  \centering
  \begin{tabular}{|c|c|}
  \hline
  Set \(\mathbb{X}\) & Projection of \(x\) onto \(\mathbb{X}\)  if \(x\notin \mathbb{X}\) \\
  \hhline{==} 
  \(\{x|\norm{x}_2\leq \alpha\}\) & \(\frac{\alpha}{\norm{x}_2}x\)\\
  \hline 
  \(\{x| l\leq x\leq u\}\) & \(\min\{\max\{x, l\}, u\}\)\\
  \hline 
  \(\{x| \langle a, x\rangle\leq \alpha, a\neq 0\}\) & \(x-(\langle a, x\rangle-\alpha)\frac{a}{\norm{a}_2}\)\\
  \hline 
  \(\{x=(y, \alpha)| \norm{y}_2\leq \alpha \}\) &
  \begin{tabular}{c}\((0, 0)\) if \(\norm{y}_2\leq -\alpha\); \\
  \(\frac{\norm{y}_2+\alpha}{2\norm{y}_2}(y, \norm{y}_2)\) otherwise \end{tabular}
  \\
  \hline
  \(\{x=(y, \alpha)| f(y)\leq \alpha\}\) & \begin{tabular}{c}\((x, f(x))\) where\\
  \(x\) solves \(y\in x+(f(x)-\alpha)\partial f(x)\)\end{tabular}\\
  \hline
  \(\{x| f(x)\leq \alpha\}\) & \begin{tabular}{c} \((I+\mu\partial f)^{-1}(x)\) where \\
  \(\mu\) solves \(f((I+\mu\partial f)^{-1}(x))=\alpha\)
  \end{tabular} \\
  \hline
  \end{tabular}
  \begin{tablenotes}
      \scriptsize
      \item Here function \(f:\mathbb{R}^n\to\mathbb{R}\) is continuous, convex and finite valued, \(\partial f\) denotes the subdifferential of \(f\), \(\max\)/\(\min\) is evaluated element-wise.
    \end{tablenotes}
  \end{threeparttable}
\end{table}

Recently there has been an increasing interest in first order primal-dual methods for MPC. Such methods solve problem \eqref{opt: QP} together with its dual problem by updating both primal and dual variables at each iteration. For example, the dual fast gradient method first updates the primal variables by optimizing the Lagrangian, then the dual variables using Nesterov's method \cite{richter2013certification,patrinos2013accelerated,giselsson2014improved}. Similarly, the Chambolle \& Pock method first updates the primal variables by optimizing the augmented Lagrangian, then the dual variables using gradient ascent with extrapolation \cite{kufoalor2014embedded}. The alternating directional method of multipliers (ADMM) first updates two copies of the primal variables by optimizing the augmented Lagrangian: one subject to \(Gz=g\), and the other subject to \(z\in\mathbb{Z}\). The dual variables then simply integrate the difference between the two copies \cite{o2013splitting,jerez2014embedded,sokoler2014input,dang2015embedded}. Compared with second order methods \cite{wang2009fast} and first order primal methods \cite{richter2011computational}, first order primal-dual methods allow both efficient per-iteration computation and general state and input constraints. 

The common challenge in implementing the aforementioned primal-dual methods is optimizing the (augmented) Lagrangian during each iteration. In general, such optimization requires either inner loop iterations that costs multiple projections onto set \(\mathbb{Z}\) \cite{kogel2011fast}, or the solution to system of linear equations that demands Ricatti recursion \cite{patrinos2011global,sokoler2014input} or pre-computing matrix inverse/decomposition \cite{giselsson2014improved,o2013splitting,jerez2014embedded}.

To address this challenge, several primal-dual methods have been developed recently, where each iteration only computes a single projection onto set \(\mathbb{Z}\), rather than minimizing the augmented Lagrangian \cite{blanchard2017sor,chambolle2016ergodic,seidman2018chebyshev}. However, these results still have their limitations. In particular, no convergence rate was proved in \cite{blanchard2017sor}. On the other hand, while \cite[Thm. 1, Thm. 4]{chambolle2016ergodic} and \cite[Thm. 1]{seidman2018chebyshev} proved the convergence rate of a non-negative duality gap function along a sequence of averaged iterates (also known as ergodic convergence), they however provide no convergence rate for the affine constraint violation, which is commonly used as a stopping criterion \cite{patrinos2013accelerated}.

We propose a novel primal-dual method for MPC problem \eqref{opt: QP}, termed \emph{proportional-integral projected gradient method}, that also computes a single projection onto set \(\mathbb{Z}\) per iteration. Our method ensures that, along a sequence of averaged iterates, both the distance to optimum and the constraint violation converges to zero at a rate of \(O(1/k)\) if the objective function is convex, where \(k\) is the iteration number. If the objective function is strongly convex, this rate can be improved to \(O(1/k^2)\) for the distance to optimum and \(O(1/k^3)\) for the constraint violation.

The rest of the paper is organized as follows. After reviewing some related work in Section~\ref{section: related}, Section~\ref{section: method} introduces our method together with its convergence guarantee. Section~\ref{section: implementation} discusses the implementation of our method on a tracking problem common in MPC. Section~\ref{section: experiments} demonstrates our results via a trajectory-planning example. Finally, Section~\ref{section: conclusion} concludes and comments on future directions. 

\paragraph*{Notation} Let \(z, y\in\mathbb{R}^n\) and \(H, P\in\mathbb{R}^{n\times n}\) with \(H=H^\top, P=P^\top\). We denote \(\langle z, y\rangle=z^\top y\), \(\norm{z}_2=\sqrt{z^\top z}\),  \(\norm{z}_H=\sqrt{z^\top H z}\). We say \(H \prec (\preceq) P\) if and only if \(P-H\) is positive (semi-)definite. The Euclidean projection onto a closed convex set \(\mathbb{Z}\subset \mathbb{R}^n\)  is denoted by \(\pi_{\mathbb{Z}}:\mathbb{R}^n\to\mathbb{Z}\) where
\[\pi_{\mathbb{Z}}[z]=\underset{z'\in\mathbb{Z}}{\argmin}\norm{z'-z}_2.\]

\section{Related work}
\label{section: related}
In this section, we briefly review some existing first order primal-dual methods for MPC. In the following, we let \(k\) and \(j\) denote iteration counters, and \(\alpha\) the step size.

\subsection{Dual fast gradient method}
Assuming \(H\succ 0\) in \eqref{opt: QP}, dual fast gradient method \cite{richter2013certification,patrinos2013accelerated,giselsson2014improved} solves problem \eqref{opt: QP} as follows
\begin{subequations}
    \begin{align}
    z^{k+1}&=\underset{z\in\mathbb{Z}}{\argmin}\,\, \textstyle \frac{1}{2}z^\top H z+h^\top z+\langle v^k, Gz\rangle,\label{eqn: Nesterov primal} \\
    w^{k+1}&=\textstyle v^k+\alpha(G z^{k+1}-g),\\
    v^{k+1}&=\textstyle w^{k+1}+ \frac{k}{k+3} (w^{k+1}-w^k).
    \end{align}\label{alg: Nesterov}
\end{subequations}
The idea is to apply Nesterov's method \cite[Sec. 2.2]{nesterov2018lectures} to the dual problem of \eqref{opt: QP}. In general, the minimization step in \eqref{eqn: Nesterov primal} can only be solved approximately using another inner loop of Nesterov's method \cite{kogel2011fast}, which iterates as follows \cite[Sec. 2.2]{nesterov2018lectures} (\(j\) denotes the inner loop iteration counter)
\begin{equation}\label{alg: Nesterov inner loop}
    \begin{aligned}
        z^{j+1}&=\textstyle \pi_{\mathbb{Z}}[y^j-\frac{1}{\lambda}(Hy^j+h+G^\top v^k)],\\
        y^{j+1}&=\textstyle z^{j+1}+\frac{\sqrt{\lambda}-\sqrt{\mu} }{\sqrt{\lambda}+\sqrt{\mu}}(z^{j+1}-z^j),
    \end{aligned}
\end{equation}
where \(0\prec \mu I\preceq H \preceq \lambda I\).

\subsection{ADMM}
One of the most popular methods for problem \eqref{opt: QP} is ADMM \cite{o2013splitting,jerez2014embedded,sokoler2014input,dang2015embedded}, which iterates as follows
\begin{subequations}
    \begin{align}
    y^{k+1}&=\underset{z:Gz=g}{\argmin}\,\, \textstyle \frac{1}{2}z^\top H z+h^\top z+ \textstyle \frac{1}{2\alpha}\norm{z+w^k-y^k}_2^2,\label{eqn: ADMM hyperplane} \\
    z^{k+1}&=\pi_{\mathbb{Z}}[y^{k+1}+w^k],\label{eqn: ADMM projection}\\
    w^k&=w^k+z^{k+1}-y^{k+1}.
    \end{align}\label{alg: admm}
\end{subequations}
Notice that ADMM solves two subproblems for primal variables: minimization of a quadratic function over a hyperplane in \eqref{eqn: ADMM hyperplane} and the projection in \eqref{eqn: ADMM projection}. The minimization in \eqref{eqn: ADMM hyperplane} is equivalent to solving the following system of linear equations for variable \(z\)
\begin{equation}\label{eqn: ADMM linear eqns}
    \begin{bmatrix}
    H+\frac{1}{\alpha} I & G^\top\\
    G & 0
    \end{bmatrix}\begin{bmatrix}
    z\\
    v
    \end{bmatrix}=\begin{bmatrix}-h-\frac{1}{\alpha}(w^k-y^k)\\
    g\end{bmatrix},
\end{equation}
which requires pre-computing either matrix inverse \cite{jerez2014embedded} or LDL decomposition \cite{o2013splitting}. If both matrix \(H\) and \(G\) are time invariant, such pre-computation only needs to be executed once. However, for time varying applications, \eg, those from nonlinear MPC \cite{kouzoupis2018recent}, such precomputation needs to be executed every time matrix \(H\) or \(G\) is updated.
\subsection{Chambolle \& Pock method}
The \emph{Chambolle \& Pock method} \cite{kufoalor2014embedded} has been used to solve problem \eqref{opt: QP} through the following iterates
\begin{subequations}
    \begin{align}
    z^{k+1}&=\underset{z\in\mathbb{Z}}{\argmin}\,\, \textstyle  \frac{1}{2}z^\top H z+h^\top z+ \frac{1}{2\alpha}\norm{z+\alpha G^\top w^k-z^k}_2^2\label{eqn: CP old primal}\\
    w^{k+1}&=w^k+\alpha (G(2z^{k+1}-z^k)-g).
    \end{align}\label{alg: CP old}
\end{subequations}
Recently, a more efficient variant of \eqref{alg: CP old} was introduced in \cite{chambolle2016ergodic}, which replaces \eqref{eqn: CP old primal} with a single projection onto \(\mathbb{Z}\). If \(H\succeq 0\) in \eqref{opt: QP}, this method achieves \(O(1/k)\) convergence rate using the following iterates \cite[Alg. 1]{chambolle2016ergodic}
\begin{subequations}
    \begin{align}
    z^{k+1}&=\pi_{\mathbb{Z}}[z^k-\alpha(Hz^k+h+G^\top w^k)]\\
    w^{k+1}&=w^k+\beta (G(2z^{k+1}-z^k)-g). \label{alg: CP const w}
    \end{align}\label{alg: CP const}
\end{subequations}
Furthermore, if \(H\succ 0\), an improved convergence rate of \(O(1/k^2)\) is possible using the following iterates \cite[Alg. 4]{chambolle2016ergodic}
\begin{subequations}
    \begin{align}
    w^{k+1}&=w^k+\beta^k (G(z^k+\gamma^k(z^k-z^{k-1}))-g),\\
    z^{k+1}&=\pi_{\mathbb{Z}}[z^k-\textstyle \frac{\alpha^k}{\mu\alpha^k+1}(Hz^k+h+G^\top w^{k+1})]
    \end{align}\label{alg: CP var}
\end{subequations}
where step sizes \((\alpha^k, \beta^k, \gamma^k)\) are computed recursively; see \cite[Sec. 5.2]{chambolle2016ergodic} for details. A different version of \eqref{alg: CP const} was introduced in \cite{blanchard2017sor}, where \eqref{alg: CP const w} is replaced with \(w^{k+1}=w^k+\beta (Gz^{k+1}-g)\). However, no convergence rate was provided in \cite{blanchard2017sor}. On the other hand, the convergence of algorithm \eqref{alg: CP const} and \eqref{alg: CP var} are both proved using a non-negative running duality gap function \cite[Thm. 1, Thm. 4]{chambolle2016ergodic}. However, these results provide no convergence rates on constraint violation amount, namely \(\norm{Gz^k-g}_2^2\).

\section{Proportional-integral projected gradient method}
\label{section: method}
We now introduce \emph{proportional-integral projected gradient method} for problem \eqref{opt: QP}. We will show that, along certain sequences of averaged iterates, our method achieves a convergence rate of \(O(1/k)\) and \(O(1/k^2)\) when matrix \(H\) is positive semi-definite and, respectively, positive definite.

The method we propose iterates as follows
\begin{subequations}
    \begin{align}
    v^k&=w^k+\beta^k(Gz^k-g)\label{eqn: pd proportional},\\
    z^{k+1}&=\pi_{\mathbb{Z}}[z^k-\alpha^k ( Hz^k+h+G^\top v^k)]\label{eqn: pd proj grad},\\
    w^{k+1}&=w^k+\beta^k (Gz^{k+1}-g)\label{eqn: pd integral}.
    \end{align}\label{alg: PI}
\end{subequations}
\begin{remark}
Notice that \eqref{eqn: pd proportional} and \eqref{eqn: pd integral} compute a proportional and, respectively, integral feedback of the affine constraints violation. Hence an intuitive interpretation of  \eqref{alg: PI} is applying projected gradient method to variable \(z\), where the gradient is corrected by a proportional-integral (PI) feedback. Similar PI feedback was also used in distributed optimization algorithms \cite{wang2010control,seidman2018chebyshev,yu2020mass,yu2020rlc}.
\end{remark}

We will prove that method \eqref{alg: PI} achieves \(O(1/k)\) and \(O(1/k^2)\) convergence rate when matrix \(H\) is positive semi-definite and, respectively, positive definite. The latter is optimal for solving problem \eqref{opt: QP} using first order methods \cite[Thm.1.1]{ouyang2019lower}.

We now prove the convergence properties of method \eqref{alg: PI}. First, we group our assumptions as follows.  

\begin{assumption}\label{asp: basic}
Suppose
\begin{enumerate}
    \item set \(\mathbb{Z}\subset\mathbb{R}^n\) is closed and convex; matrix \(H\in\mathbb{R}^{n\times n}\) is symmetric, matrix \(G\in\mathbb{R}^{m\times n}\) has full row rank; there exists \(\mu, \lambda, \sigma \in\mathbb{R}\) with \(0\leq \mu\leq \lambda\) and \(\sigma\geq 0\) such that \(\mu I\preceq H\preceq \lambda I\) and \(G^\top G\preceq \sigma I\).
    \item there exists \(z^\star\in\mathbb{R}^n\) and \(w^\star\in\mathbb{R}^m\) such that
\begin{subequations}\label{eqn: KKT}
\begin{align}
Gz^\star=g, \enskip  z^\star&\in \mathbb{Z},\label{eqn: KKT primal}\\
\langle Hz^\star+h+G^\top w^\star, z-z^\star\rangle &\geq 0, \enskip \forall z\in\mathbb{Z}.\label{eqn: KKT dual}
\end{align}
\end{subequations}
\end{enumerate}
\end{assumption}

\begin{remark}
Equation \eqref{eqn: KKT} gives the Karush–Kuhn–Tucker conditions of problem \eqref{opt: QP}. Under the Slater condition for equalities, equation \eqref{eqn: KKT} holds if and only if  \(z^\star\) is an optimal solution for problem \eqref{opt: QP}; see \cite[Thm.3.1.27]{nesterov2018lectures}.
\end{remark}

We will use the following result on  Euclidean projection.
\begin{lemma}\cite[Lemma. 2.2.7]{nesterov2018lectures}\label{lem: projection} If set \(\mathbb{Z}\subset\mathbb{R}^n\) is closed and convex, then
\[\langle \pi_{\mathbb{Z}}[z]-z, z'-\pi_{\mathbb{Z}}[z]\rangle\geq 0,\enskip \forall z\in\mathbb{R}^n,z'\in\mathbb{Z}.\]
\end{lemma}

The following lemma shows the key property of any two consecutive iterations generated by method \eqref{alg: PI}. 

\begin{lemma}\label{lem: iteration}
Suppose Assumption~\ref{asp: basic} holds and sequence \(\{z^k, w^k\}\) is generated by \eqref{alg: PI}. If \(\lambda+\sigma\beta^k=\frac{1}{\alpha^k}\) for all \(k\geq 1\), then
\begin{equation*}
\begin{aligned}
    &\textstyle 
    \frac{\beta^k}{2}\norm{Gz^k-g}_2^2+\frac{1}{2}\norm{z^{k+1}-z^\star}_H^2\\
    \leq &\textstyle \frac{1}{2}(\frac{1}{\alpha^k}-\mu)\norm{z^k-z^\star}_2^2+ \frac{1}{2\beta^k}\norm{w^k-w^\star}_2^2\\
    &\textstyle-\frac{1}{2\alpha^k}\norm{z^{k+1}-z^\star}_2^2-\frac{1}{2\beta^k}\norm{w^{k+1}-w^\star}_2^2.
\end{aligned}    
\end{equation*}
\end{lemma}
\begin{proof}
First, applying Lemma~\ref{lem: projection} to \eqref{eqn: pd proj grad} gives
\begin{equation}\label{eqn: projection optimality}
\begin{aligned}
0\leq &\textstyle \frac{1}{\alpha^k}\langle z^{k+1}-z^k, z^\star-z^{k+1}\rangle+\beta^k\langle Gz^k-g, g-Gz^{k+1}\rangle\\
&+\langle Hz^k+h, z^\star-z^{k+1}\rangle+\langle  w^k, G(z^\star-z^{k+1})\rangle,
\end{aligned}
\end{equation}
where we also used \eqref{eqn: pd proportional} and \eqref{eqn: KKT primal}. 
Next, \eqref{eqn: KKT dual} implies that
\begin{equation}\label{eqn: optimality}
    0\leq -\langle H z^\star+h, z^\star-z^{k+1}\rangle-\langle w^\star, G(z^\star-z^{k+1})\rangle.
\end{equation}
In addition, one can directly verify the following four identities, which can be interpreted as instances of the \emph{law of cosines}; see Fig.~\ref{fig: los} for an illustration.
\begin{equation}\label{eqn: LOS primal}
     \begin{aligned}
     &\langle z^{k+1}-z^k, z^\star-z^{k+1}\rangle\\
     =&\textstyle \frac{1}{2}\norm{z^k-z^\star}_2^2-\frac{1}{2}\norm{z^{k+1}-z^\star}_2^2-\frac{1}{2}\norm{z^{k+1}-z^k}_2^2,
     \end{aligned}
\end{equation}
\begin{equation}\label{eqn: LOS augmented}
    \begin{aligned}
    &\langle Gz^k-g, g-Gz^{k+1} \rangle\\
    =&\textstyle \frac{1}{2}\norm{G(z^{k+1}-z^k)}_2^2-\frac{1}{2}\norm{Gz^k-g}_2^2-\frac{1}{2}\norm{Gz^{k+1}-g}_2^2,
    \end{aligned}
\end{equation}
\begin{equation}\label{eqn: LOS H}
    \begin{aligned}
        &\langle H^{\frac{1}{2}}(z^k-z^\star), H^{\frac{1}{2}}(z^\star-z^{k+1})\rangle\\
        =&\textstyle \frac{1}{2}\norm{z^{k+1}-z^k}_H^2-\frac{1}{2}\norm{z^k-z^\star}_H^2-\frac{1}{2}\norm{z^{k+1}-z^\star}_H^2,
    \end{aligned}
\end{equation}
\begin{equation}\label{eqn: LOS dual}
    \begin{aligned}
    &\textstyle\frac{1}{2}\norm{w^{k+1}-w^\star}^2_2-\frac{1}{2}\norm{w^k-w^\star}_2^2\\
    =&\textstyle \langle w^k-w^\star, w^{k+1}-w^k\rangle+\frac{1}{2}\norm{w^{k+1}-w^k}_2^2\\
    =&\beta^k\langle w^k-w^\star, G(z^{k+1}-z^\star) \rangle+\textstyle \frac{(\beta^k)^2}{2}\norm{Gz^{k+1}-g}_2^2,
    \end{aligned} 
\end{equation}
where matrix \(H^{\frac{1}{2}}\) in \eqref{eqn: LOS H} is the positive semi-definite square root of \(H\), and the last step in \eqref{eqn: LOS dual} is due to \eqref{eqn: pd integral} and \eqref{eqn: KKT primal}.
\begin{figure}[ht]
    \centering
    \begin{tikzpicture}[scale=0.5]
\coordinate (A) at (0,0);
\coordinate (B) at (2,0);
\coordinate (C) at (1,1.5);

\node at (1, 0.9) {$\theta$};

\draw (A) -- (B) node[midway,below] {$c$}
   -- (C) node[midway,right] {$a$}
   -- (A) node[midway,left] {$b$};

\node at (8, 0.75) {\(-ab\cos\theta=\frac{1}{2}c^2-\frac{1}{2}a^2-\frac{1}{2}b^2\)};

\end{tikzpicture}
    \caption{The law of cosines.}
    \label{fig: los}
\end{figure}
Further, the assumption that \(0\preceq \mu I\preceq H\preceq \lambda I\) and \(G^\top G\preceq \sigma I\) implies the following
\begin{subequations}
\begin{align}
    \textstyle \frac{\mu}{2}\norm{z^k-z^\star}_2^2\leq & \textstyle\frac{1}{2}\norm{z^k-z^\star}_H^2,\label{eqn: quad lower bound}\\
    \textstyle \frac{1}{2}\norm{z^{k+1}-z^k}_H^2\leq & \textstyle \frac{\lambda}{2}\norm{z^{k+1}-z^k}_2^2,\label{eqn: quad upper bound}\\
    \textstyle \frac{1}{2}\norm{G(z^{k+1}-z^k)}_2^2\leq &\textstyle  \frac{\sigma}{2}\norm{z^{k+1}-z^k}_2^2.\label{eqn: quad upper bound 2}
    \end{align}
\end{subequations}
Finally, summing up together \eqref{eqn: projection optimality}, \eqref{eqn: optimality},  \(\frac{1}{\alpha^k}\times\)\eqref{eqn: LOS primal}, \(\beta^k\times\)\eqref{eqn: LOS augmented}, \eqref{eqn: LOS H},  \(\frac{1}{\beta^k}\times\)\eqref{eqn: LOS dual}, \eqref{eqn: quad lower bound}, \eqref{eqn: quad upper bound} and \(\beta^k\times \)\eqref{eqn: quad upper bound 2}, then using the assumption that \(\lambda+\sigma\beta^k=\frac{1}{\alpha^k}\), we obtain the desired result.
\end{proof}

With the above lemma, we are now ready to prove the convergence properties of method \eqref{alg: PI} under different assumptions on matrix \(H\). We start with the case where matrix \(H\) is only positive semi-definite, \ie, \(\mu=0\) in Assumption~\ref{asp: basic}. The following theorem shows that, along a sequence of averaged iterates generated by \eqref{alg: PI}, both the quadratic distance to optimum and constraint violation converge to zero at the rate of \(O(1/k)\).
\begin{theorem}\label{thm: convex}
Suppose Assumption~\ref{asp: basic} hold with \(\mu=0\), and sequence \(\{v^k, z^k, w^k\}\) is generated by \eqref{alg: PI} with \(\alpha^k=\frac{1}{\beta\sigma+\lambda}\) and \(\beta^k=\beta\) for some \(\beta>0\) and all \(k\geq 1\). Let \(V^1=\frac{1}{2\alpha }\norm{z^1-z^\star}_2^2+\frac{1}{2\beta}\norm{w^1-w^\star}_2^2\), then
\begin{equation*}
 \textstyle \frac{1}{2}\norm{G\hat{z}^k-g}_2^2
    \leq  \frac{1}{\beta k}V^1, \enskip \textstyle \frac{1}{2}\norm{\tilde{z}^k-z^\star}_H^2
    \leq \textstyle \frac{1}{k}V^1.
\end{equation*}
where \(\hat{z}^k=\frac{1}{k}\sum_{j=1}^k z^j\) and \(\tilde{z}^k=\frac{1}{k}\sum_{j=1}^k z^{j+1}\).
\end{theorem}
\begin{proof}
With this choice of \(\alpha^k\) and \(\beta^k\), the inequality in Lemma~\ref{lem: iteration} becomes the following: for all \(j\geq 1\),
\begin{equation*}
    \textstyle \frac{\beta}{2}\norm{Gz^j-g}_2^2+\frac{1}{2}\norm{z^{j+1}-z^\star}_H^2\leq V^j-V^{j+1},
\end{equation*}
where \(V^j= \frac{1}{2\alpha}\norm{z^j-z^\star}_2^2+\frac{1}{2\beta}\norm{w^j-w^\star}_2^2\). Summing up this inequality for \(j=1, \ldots, k\) gives
\begin{equation*}
\begin{aligned}
      &\textstyle \sum_{j=1}^k\big(\frac{\beta}{2}\norm{Gz^j-g}_2^2+\frac{1}{2}\norm{z^{j+1}-z^\star}_H^2\big)\\
      \leq &V^1-V^{k+1}\leq V^1
\end{aligned}
\end{equation*}
where the last step is because \(V^{k+1}\geq 0\). Hence
\begin{equation*}
    \begin{aligned}
      \textstyle \frac{\beta}{2}\sum_{j=1}^k\norm{Gz^j-g}_2^2
      \leq V^1, \enskip   \frac{1}{2}\sum_{j=1}^k\norm{z^{j+1}-z^\star}_H^2
      \leq V^1, 
    \end{aligned}
\end{equation*}
Finally, applying Jensen's inequality to the above two inequalities gives the desired results.
\end{proof}

If matrix \(H\) is positive definite, \ie, \(\mu>0\) in Assumption~\ref{asp: basic}, then, along two different sequences of averaged iterates, the \(O(1/k)\) rate in Theorem~\ref{thm: convex} can be improved to \(O(1/k^2)\) for the quadratic distance to optimum and \(O(1/k^3)\) for constraint violation. 

\begin{theorem}\label{thm: strongly convex}
Suppose Assumption~\ref{asp: basic} hold with \(\mu>0\), and sequence \(\{v^k, z^k, w^k\}\) is generated by \eqref{alg: PI} with \(\alpha^k=\frac{2}{(k+1)\mu+2\lambda}\), \(\beta^k= \frac{(k+1)\mu}{2\sigma}\) for all \(k\geq 1\). Let \(V^1=\frac{1}{2(\mu+\lambda) }\norm{z^1-z^\star}_2^2+\frac{\sigma}{2\mu}\norm{w^1-w^\star}_2^2\), then
\begin{equation*}
\begin{aligned}
    \textstyle \frac{1}{2}\norm{G\hat{z}^k-g}_2^2\leq &\textstyle  \frac{12\lambda\sigma}{\mu^2 k(k^2+6k+11)}V^1,\\
    \textstyle \frac{1}{2}\norm{\tilde{z}^k-z^\star}_H^2
    \leq &\textstyle \frac{4\lambda}{\mu k(k+5)}V^1,
\end{aligned}
\end{equation*}
where \(\hat{z}^k=\textstyle \frac{3}{k(k^2+6k+11)}\sum_{j=1}^k (j+1)(j+2) z^j\) and \(\tilde{z}^k=\textstyle \frac{2}{k(k+5)}\sum_{j=1}^k (j+2)z^{j+1}\). 
\end{theorem}
\begin{proof}
With this choice of \(\alpha^k\) and \(\beta^k\), the inequality in Lemma~\ref{lem: iteration} becomes the following: for all \(j\geq 1\),
\begin{equation}\label{eqn: thm2 eqn1}
\begin{aligned}
    &\textstyle 
    \frac{(j+1)\mu}{4\sigma}\norm{Gz^j-g}_2^2+\frac{1}{2}\norm{z^{j+1}-z^\star}_H^2\\
    \leq &\textstyle \frac{1}{2}(\frac{1}{\alpha^j}-\mu)\norm{z^j-z^\star}_2^2+\frac{1}{2\beta^j}\norm{w^j-w^\star}_2^2-V^{j+1},
\end{aligned}    
\end{equation}
where \(V^j= \frac{1}{2\alpha^{j-1}}\norm{z^j-z^\star}_2^2+\frac{1}{2\beta^{j-1}}\norm{w^j-w^\star}_2^2\). Let \(\kappa=\lambda/\mu\geq 1\), then it is straightforward to verify the following 
\begin{equation}\label{eqn: thm2 eqn2}
\begin{aligned}
    \textstyle (\frac{1}{\alpha^j}-\mu)(j+2\kappa)= &\textstyle  \frac{1}{\alpha^{j-1}}(j+2\kappa-1),\\
    \textstyle \frac{1}{\beta^j}(j+2\kappa)\leq &\textstyle \frac{1}{\beta^{j-1}}(j+2\kappa-1).
\end{aligned}
\end{equation}
Hence multiplying \eqref{eqn: thm2 eqn1} with \((j+2\kappa)\) and substituting in \eqref{eqn: thm2 eqn2} we can show
\begin{equation*}
\begin{aligned}
     &\textstyle \frac{(j+1)(j+2\kappa)\mu}{4\sigma}\norm{Gz^j-g}_2^2+\frac{j+2\kappa}{2}\norm{z^{j+1}-z^\star}_H^2\\
    \leq &(j+2\kappa-1)V^j-(j+2\kappa)V^{j+1}.
\end{aligned}
\end{equation*}
Summing up this inequality for \(j=1, 2, \ldots, k\) gives
\begin{equation*}
\begin{aligned}
    &\textstyle \sum_{j=1}^k\big(\frac{(j+1)(j+2\kappa)\mu}{4\sigma}\norm{Gz^j-g}_2^2+\frac{j+2\kappa}{2}\norm{z^{j+1}-z^\star}_H^2\big)\\
    \leq &2\kappa V^1-(k+2\kappa)V^{k+1}\leq 2\kappa V^1.
\end{aligned}
\end{equation*}
where the last step is because \(V^{k+1}\geq 0\). Since \(\kappa\geq 1\), the above inequality implies the following
\begin{equation*}
    \begin{aligned}
    \textstyle \sum_{j=1}^k\frac{(j+1)(j+2)\mu}{4\sigma}\norm{Gz^j-g}_2^2\leq & 2\kappa V^1,\\
      \textstyle \sum_{j=1}^k\frac{j+2}{2}\norm{z^{j+1}-z^\star}_H^2\leq &2\kappa V^1.
    \end{aligned}
\end{equation*}
Finally, applying Jensen's inequality to the above two inequalities and using \(\kappa=\lambda/\mu\) gives the desired results. 
\end{proof}

Notice that Theorem~\ref{thm: strongly convex}, in contrast to Theorem~\ref{thm: convex}, uses a weighted average of iterates similar to those in subgradient method \cite{lacoste2012simpler} and accelerated ADMM \cite{xu2017accelerated}.

\begin{remark}
Theorem~\ref{thm: strongly convex} not only establishes the \(O(1/k^2)\) convergence rate of the quadratic distance to the optimum, but also shows the \(O(1/k^3)\) convergence rate of the constraint violation. To our best knowledge, the latter has never been proven for any existing methods that achieve \(O(1/k^2)\) convergence rate \cite{richter2013certification,patrinos2013accelerated,giselsson2014improved,chambolle2016ergodic,blanchard2017sor,seidman2018chebyshev}.
\end{remark}

\section{Efficient implementation for MPC}
\label{section: implementation}
In this section, we provide the pseudocode implementation of method \eqref{alg: PI} for the following tracking problem
\begin{equation}
\begin{array}{ll}
    \underset{\{u_{t-1}, x_t\}_{t=1}^T}{\mbox{minimize}} &  \frac{1}{2}\sum_{t=1}^{T} \norm{x_t-y_t}_{Q_t}^2+\frac{1}{2}\sum_{t=0}^{T-1} \norm{u_t}_{R_t}^2\\
    \mbox{subject to}  & x_t=A_{t-1}x_{t-1}+B_{t-1}u_{t-1}, \\
    & u_{t-1}\in \mathbb{U}_{t-1},\,x_t\in \mathbb{X}_t, \enskip 1\leq t\leq T. 
\end{array}
\label{opt: MPC}
\end{equation}
where, for all \(1\leq t\leq T\): closed convex sets \(\mathbb{X}_t\subset \mathbb{R}^{n_x}\) and \(\mathbb{U}_{t-1}\subset\mathbb{R}^{n_u}\) describe feasible sets for state variable \(x_t\) and, respectively, input variable \(u_{t-1}\); \(A_{t-1}\in\mathbb{R}^{n_x\times n_x}, B_{t-1}\in\mathbb{R}^{n_x\times n_u}\) describe the linear dynamics of the plant; \(y_t\) gives the reference value for \(x_t\). 

We first rewrite problem \eqref{opt: MPC} as a special case of problem \eqref{opt: QP} by defining the following
\begin{equation}\label{eqn: lifting}
\begin{aligned}
    z=&[u_0^\top, x_1^\top, \ldots, u_{T-1}^\top,x_T^\top]^\top, \enskip \mathbb{Z}=\textstyle \prod_{t=1}^T(\mathbb{U}_{t-1}\times \mathbb{X}_t),\\
    H=&\blkdiag(R_0,Q_1, \ldots, R_{T-1},Q_T),\\
    h=& [0^\top,-y_1^\top Q_1, \ldots, 0^\top, -y_T^\top Q_T]^\top,\\
    G=&\begin{bsmallmatrix}
   -B_0 & I    &  &     &      &   &     \\
       & -A_1  & -B_1   &I     &      & &        \\
        & & \ddots & \ddots& \ddots & &\\
     && &  & -A_{T-1}  & -B_{T-1} & I     \\
    \end{bsmallmatrix}\\
g=&[x_0^\top A_0^\top, 0^\top, 
    \cdots,
    0^\top]^\top.
\end{aligned}
\end{equation}

We are now ready to implement \eqref{alg: PI} for problem \eqref{opt: MPC}. We partition variables \(w\) and \(v\) as follows
\begin{equation}
    v =[v_1^\top, v_2, \ldots, v_T^\top]^\top, \enskip w =[w_1^\top, w_2, \ldots, w_T^\top]^\top,
\end{equation}
where \(v_t, w_t\in\mathbb{R}^{n_x}\) corresponds to constraint \(x_t=A_{t-1}x_{t-1}+B_{t-1}u_{t-1}\) for \(1\leq t \leq T\). In addition, the separable structure of set \(\mathbb{Z}\) defined by \eqref{eqn: lifting} allows separable computation of its Euclidean projection. Based on these observations, we implement algorithm \eqref{alg: PI} for problem \eqref{opt: MPC} in Algorithm~\ref{code: PI}, where we introduce dummy parameters \(A_Tv_{T+1}\equiv 0\) to simplify our notation. Notice that updates of variables corresponding to different value of \(t\) can be executed in parallel, hence the algorithm run-time can be almost independent of horizon \(T\).

\begin{algorithm}
\caption{PI projected gradient method}
\label{code: PI}
\begin{algorithmic}
\Require \(x_0\); \(\mathbb{X}_t, \mathbb{U}_{t-1}\),  \(Q_t, y_t, R_{t-1},A_{t-1}, B_{t-1}\) for all \(1\leq t\leq T\). Initialize \(k=1\), \(u_{t-1}, x_t, w_t\) for all \(1\leq t\leq T\); let \(A_Tv_{T+1}\equiv 0\).
\While{\(k\leq k_{\max}\)}
\State \(k\gets k+1\)
\State For all \(1\leq t\leq T\):
\State \(v_t\gets w_t+\beta^k (x_t-A_{t-1}x_{t-1}-B_{t-1}u_{t-1})\)
\State \(u_{t-1}\gets \pi_{\mathbb{U}_{t-1}} [u_{t-1}-\alpha^k(R_{t-1}u_{t-1}-B_{t-1}^\top v_t)]\)
\State \(x_t\gets \pi_{\mathbb{X}_t}[x_t-\alpha^k(Q_t(x_t-y_t)+ v_t-A_t^\top v_{t+1})]\)
\State \(w_t\gets w_t +\beta^k (x_t-A_{t-1}x_{t-1}-B_{t-1}u_{t-1})\)
\EndWhile
\Ensure \(\{u_0, x_1, \ldots, u_{T-1}, x_T\}\)
\end{algorithmic}
\end{algorithm}

\section{Numerical Examples}
\label{section: experiments}
In this section we compare our method against the existing methods reviewed in Section~\ref{section: related} over a trajectory-planning problem with keep-out-zone constraints, where all parameters are chosen as unit-less for simplicity. 

\begin{figure}[ht]
    \centering
    \begin{tikzpicture}[scale=0.5]

\filldraw[color=red, fill=red!10] (-2,0) arc (180:0:2);
\draw[->] (-0.43,2.46) arc (100:80:2.5);

\node at (-5,0.5)[circle,fill,inner sep=1pt]{};
\node at (4,0.5)[circle,fill,inner sep=1pt]{};

\filldraw[color=blue!10, fill=blue!10, rotate around={150:(0,0)}] (2,-1.8) rectangle (3.5, 1.8) node[pos=.5, rotate around={60:(0,0)}] {\scriptsize \color{blue} rotating halfspace};

\draw[color=blue, rotate around={150:(0,0)}] (2, -1.8) -- (2, 1.8);

\node[above= 0mm of {(0, 0)}]{\scriptsize  \color{red} keep-out-zone};
\node[above= 1mm of {(-5, 0.5)}]{\footnotesize initial};
\node[above= 1mm of {(4, 0.5)}]{\footnotesize target};

\end{tikzpicture}
    \caption{Trajectory-planning with rotating halfspace constraint.}
    \label{fig: path planning}
\end{figure}
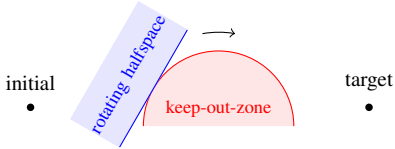

We consider a trajectory-planning (finite horizon optimal control) problem where the goal is to track a beeline trajectory from initial to target position while avoiding collision with a circular keep-out-zone; see Fig.~\ref{fig: path planning} for an illustration. Here, this problem is an instant of an MPC problem, which is solved repetitively as new state information becomes available. The dynamics of the system is modeled as a double integrator with sampling time \(0.5\) s. The system is subject to \(\ell_2\) norm constraints on its velocity \(q\in\mathbb{R}^2\) and acceleration input \(u\in\mathbb{R}^2\). In addition, a rotating half-space constraint is imposed on its position \(p\in\mathbb{R}^2\), which convexifies the keep-out-zone; see \cite{zagaris2018model} for a detailed discussion. We model this tracking problem as a special case of problem \eqref{opt: MPC} with the following choice of parameters:
\begin{subequations}
\begin{align}
    &A_{t-1}=\begin{bsmallmatrix}
    1 & 0 & 0.5 & 0\\
    0 & 1 & 0 & 0.5\\
    0 & 0 & 1 & 0\\
    0 & 0 & 0 & 1
    \end{bsmallmatrix}, \: B_{t-1}=\begin{bsmallmatrix}
    0.125 & 0\\
    0 & 0.125\\
    0.5 & 0\\
    0 & 0.5
    \end{bsmallmatrix}, \label{eqn: dynamics}\\
    &Q_t=\diag (1, 0.5, 1, 0.5), \enskip  R_{t-1}=\diag(1, 0.5), \label{eqn: cost}\\
    &\mathbb{X}_t=\left\{\:
     x=\begin{bsmallmatrix}
     p\\
     q
     \end{bsmallmatrix}\: 
    \middle\vert\:
     \begin{bsmallmatrix}
     -\cos(\theta t)\\
     \sin(\theta t)
     \end{bsmallmatrix}^\top p\geq 2,\: \norm{q}_2\leq 0.25
  \right\}, \label{eqn: state constraints}\\
  &\mathbb{U}_{t-1}=\{u|\norm{u}_2\leq 0.1\}, \enskip x_0=\begin{bmatrix}
  -2.5 & 0.6 & 0 & 0
  \end{bmatrix}^\top ,
  \end{align}
\end{subequations}
for \(1\leq t\leq T\), where \(\theta=0.063\) in \eqref{eqn: state constraints} is a constant rotation rate \cite{zagaris2018model}. Note that \(Q_t\) and \(R_t\) in \eqref{eqn: cost} are diagonal but not identity, which is common in practice. The reference trajectory \(\{y_t\}_{t=1}^T\) in \eqref{opt: MPC} is chosen as a beeline trajectory from initial position \((-2.5, 0.6)\) to target position \((2.9, 0.3)\) without considering the position constraint on \(p\) in \eqref{eqn: state constraints}.

We compare our method against all the other methods reviewed in Section~\ref{section: related}. In terms of step sizes: for our method \eqref{alg: PI}, we choose \(\alpha^k\) and \(\beta^k\) according to Theorem~\ref{thm: convex} and Theorem~\ref{thm: strongly convex} for constant and, respectively, varying step sizes; for dual fast gradient method, we choose \(\alpha\) according to \cite[Thm.1]{richter2013certification}; for Chambolle \& Pock method (C \& P), we choose the constant step sizes in \eqref{alg: CP const} according to \cite[Rem. 1]{chambolle2016ergodic}, and the varying step sizes in \eqref{alg: CP var} according to the ``optimal rule'' in  \cite[Sec. 5.2]{chambolle2016ergodic}; for ADMM, we choose \(\alpha=2\) as suggested in \cite{jerez2014embedded}. In addition, the inner loop iterations used by each iteration of method \eqref{alg: Nesterov} are warm-started using results from the last outer iteration and terminated if \(\norm{z^{j+1}-z^j}_2/\norm{z^j}_2\leq \epsilon_{\text{inner}}\), where \(\epsilon_{\text{inner}}\) is chosen between \(0.1\%\) and \(0.01\%\). 

We summarize our results as follows. Fig.~\ref{fig: single sample} shows the convergence over iterations of different algorithms with same initialization for \(T=25\), where \(z^\star\) is computed using  ECOS \cite{domahidi2013ecos} together with JuMP \cite{dunning2017jump}. 
Fig.~\ref{fig: statistics} shows the computation costed by different algorithms for \(T=\{5, 15, 25, 35, 45\}\) to reach the tolerance for constraint violation (we use \(\ell_\infty\)-norm since it measures the maximum pointwise constraint violation along the trajectory), where each data point is averaged over \(200\) independent experiments using initialization sampled from standard normal distribution. Note that we omitted method \eqref{alg: CP const} and \eqref{alg: PI} with constant step sizes in Fig.~\ref{fig: fine epsilon} due to their slow convergence. 

In these simulations, our method with varying step sizes (var.) outperforms the others. Our method with constant step sizes (const.) converges slower, but is almost identically to Chambolle \& Pock method (C \& P) with constant step sizes (const.), since they do not exploit the strong convexity of the objective functions.

\begin{figure}[ht!]
\centering
  \begin{subfigure}{0.48\columnwidth}
  \includegraphics[trim=0.2cm 0.1cm 0.2cm 0.1cm,width=\textwidth]{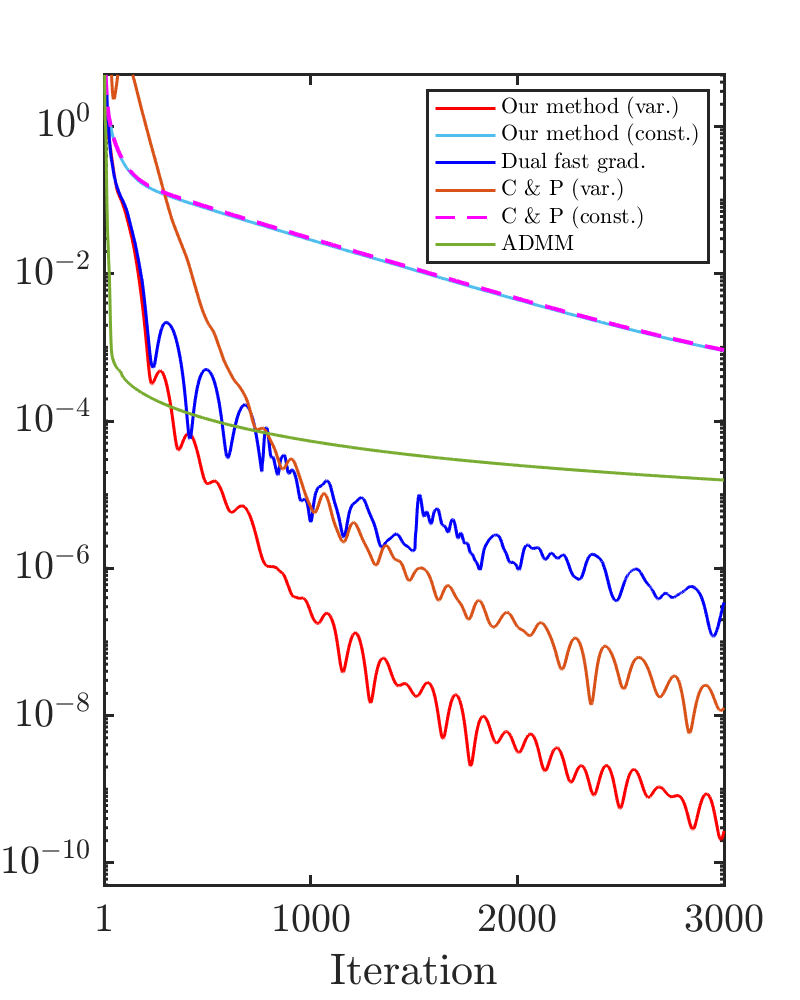}
  \caption{\(\norm{z-z^\star}_2^2\).}
  \end{subfigure}
  \begin{subfigure}{0.48\columnwidth}
  \includegraphics[trim=0.2cm 0.1cm 0.2cm 0.1cm,width=\textwidth]{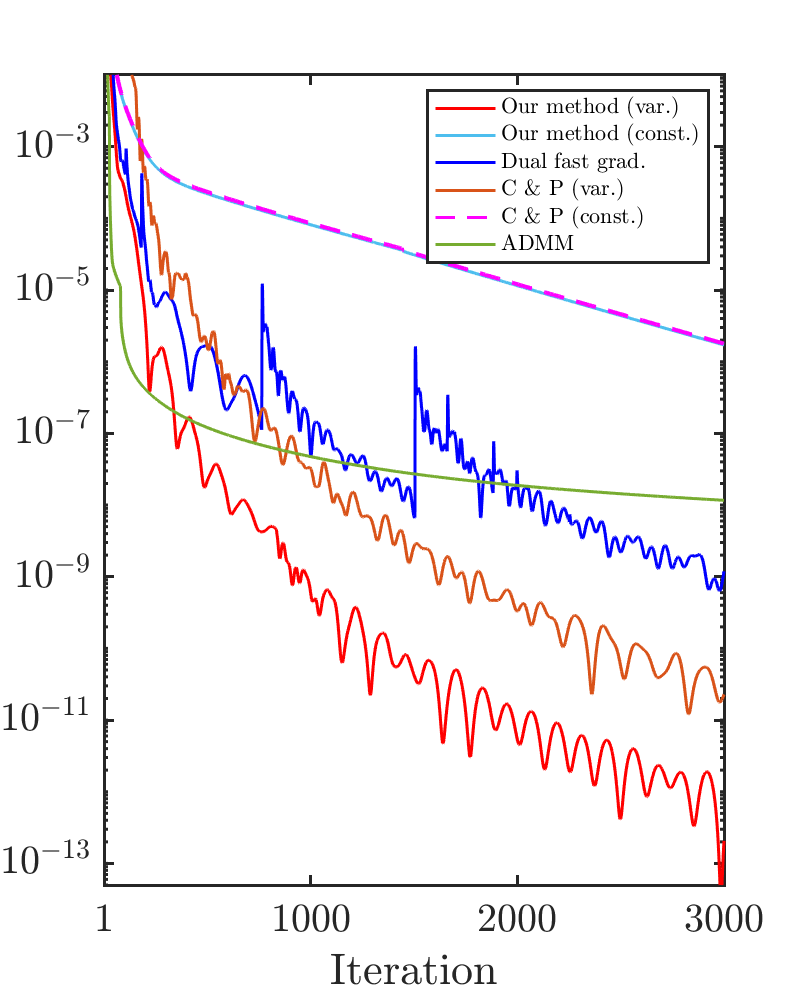}
  \caption{ \(\norm{Gz-g}_2^2\).} 
  \end{subfigure} 
  \caption{Convergence over iterations for \(T=25\).}
  \label{fig: single sample}
\end{figure}

\begin{figure}[ht!]
\centering
  \begin{subfigure}{0.48\columnwidth}
  \includegraphics[trim=0.1cm 0.1cm 0.1cm 0.1cm,width=\textwidth]{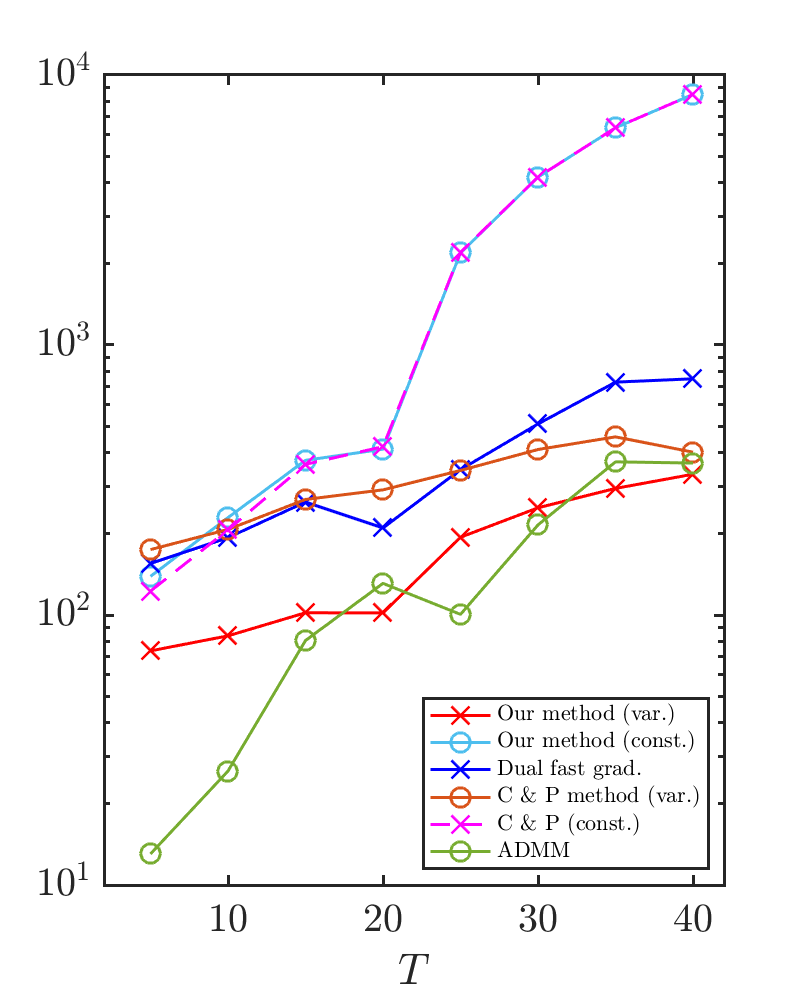}
  \caption{\(\epsilon=10^{-3}\).}
  \label{fig: coarse epsilon}
  \end{subfigure} 
  \begin{subfigure}{0.48\columnwidth}
  \includegraphics[trim=0.1cm 0.1cm 0.1cm 0.1cm,width=\textwidth]{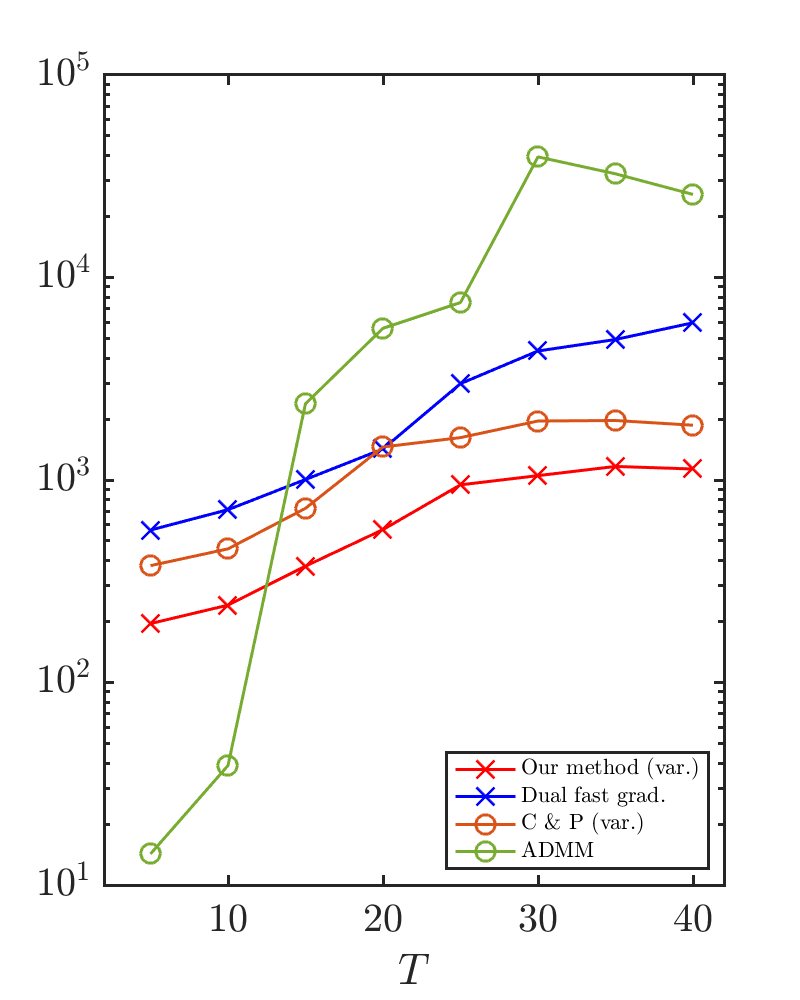}
  \caption{\(\epsilon=10^{-5}\).}
  \label{fig: fine epsilon}
  \end{subfigure}
  \caption{Number of projection $\pi_\mathbb{Z}[\cdot]$ costed to reach condition $\|Gz -g\|_\infty\leq \epsilon$. Each data point is averaged over 200 simulations using random initialization.}
  \label{fig: statistics}
\end{figure}

\section{Conclusion}\label{section: conclusion}
We introduced a novel first order primal-dual method for MPC, which uses a single projection onto the state and input constraint set per-iteration. We prove the convergence rate of both the distance to optimum and the constraint violation along different sequences of averaged iterates. Our method not only enjoys simple interpretation based on PI feedback, but also outperforms existing methods in numerical experiments.
Future directions include real-time implementation, faster empirical convergence using preconditioning, and other control-inspired optimization algorithms.




\bibliographystyle{IEEEtran}
\bibliography{IEEEabrv,reference}

\begin{thebibliography}{10}
\providecommand{\url}[1]{#1}
\csname url@rmstyle\endcsname
\providecommand{\newblock}{\relax}
\providecommand{\bibinfo}[2]{#2}
\providecommand\BIBentrySTDinterwordspacing{\spaceskip=0pt\relax}
\providecommand\BIBentryALTinterwordstretchfactor{4}
\providecommand\BIBentryALTinterwordspacing{\spaceskip=\fontdimen2\font plus
\BIBentryALTinterwordstretchfactor\fontdimen3\font minus
  \fontdimen4\font\relax}
\providecommand\BIBforeignlanguage[2]{{%
\expandafter\ifx\csname l@#1\endcsname\relax
\typeout{** WARNING: IEEEtran.bst: No hyphenation pattern has been}%
\typeout{** loaded for the language `#1'. Using the pattern for}%
\typeout{** the default language instead.}%
\else
\language=\csname l@#1\endcsname
\fi
#2}}

\bibitem{mayne2000constrained}
D.~Q. Mayne, J.~B. Rawlings, C.~V. Rao, and P.~O. Scokaert, ``Constrained model
  predictive control: Stability and optimality,'' \emph{Automatica}, vol.~36,
  no.~6, pp. 789--814, 2000.

\bibitem{mayne2014model}
D.~Q. Mayne, ``Model predictive control: Recent developments and future
  promise,'' \emph{Automatica}, vol.~50, no.~12, pp. 2967--2986, 2014.

\bibitem{qin2003survey}
S.~J. Qin and T.~A. Badgwell, ``A survey of industrial model predictive control
  technology,'' \emph{Control. Eng. Pract.}, vol.~11, no.~7, pp. 733--764,
  2003.

\bibitem{kouro2008model}
S.~Kouro, P.~Cort{\'e}s, R.~Vargas, U.~Ammann, and J.~Rodr{\'\i}guez, ``Model
  predictive control—-a simple and powerful method to control power
  converters,'' \emph{IEEE Trans. Ind. Electron.}, vol.~56, no.~6, pp.
  1826--1838, 2008.

\bibitem{eren2017model}
U.~Eren, A.~Prach, B.~B. Ko{\c{c}}er, S.~V. Rakovi{\'c}, E.~Kayacan, and
  B.~A{\c{c}}{\i}kme{\c{s}}e, ``Model predictive control in aerospace systems:
  Current state and opportunities,'' \emph{J. Guid. Control Dyn.}, vol.~40,
  no.~7, pp. 1541--1566, 2017.

\bibitem{ullmann2011fiordos}
F.~Ullmann, ``{FiOrdOs}: A {M}atlab toolbox for {C}-code generation for first
  order methods,'' Master's thesis, Dept. Inf. Technol. Elect. Eng., ETH
  Zurich, 2011.

\bibitem{richter2011computational}
S.~Richter, C.~N. Jones, and M.~Morari, ``Computational complexity
  certification for real-time {MPC} with input constraints based on the fast
  gradient method,'' \emph{IEEE Trans. Automat. Control}, vol.~57, no.~6, pp.
  1391--1403, 2011.

\bibitem{richter2013certification}
------, ``Certification aspects of the fast gradient method for solving the
  dual of parametric convex programs,'' \emph{Math. Meth. Oper. Res.}, vol.~77,
  no.~3, pp. 305--321, 2013.

\bibitem{jerez2014embedded}
J.~L. Jerez, P.~J. Goulart, S.~Richter, G.~A. Constantinides, E.~C. Kerrigan,
  and M.~Morari, ``Embedded online optimization for model predictive control at
  megahertz rates,'' \emph{IEEE Trans. Automat. Control}, vol.~59, no.~12, pp.
  3238--3251, 2014.

\bibitem{bauschke1996projection}
H.~H. Bauschke, ``Projection algorithms and monotone operators,'' Ph.D.
  dissertation, Dept. Math. and Statist., Simon Fraser Univ., 1996.

\bibitem{ghaemi2009integrated}
R.~Ghaemi, J.~Sun, and I.~V. Kolmanovsky, ``An integrated perturbation analysis
  and sequential quadratic programming approach for model predictive control,''
  \emph{Automatica}, vol.~45, no.~10, pp. 2412--2418, 2009.

\bibitem{bonalli2019gusto}
R.~Bonalli, A.~Cauligi, A.~Bylard, and M.~Pavone, ``Gusto: Guaranteed
  sequential trajectory optimization via sequential convex programming,'' in
  \emph{Int. Conf. Robot. Autom.}\hskip 1em plus 0.5em minus 0.4em\relax IEEE,
  2019, pp. 6741--6747.

\bibitem{mao2016successive}
Y.~Mao, M.~Szmuk, and B.~A{\c{c}}{\i}kme{\c{s}}e, ``Successive convexification
  of non-convex optimal control problems and its convergence properties,'' in
  \emph{IEEE Conf. Decision Control}.\hskip 1em plus 0.5em minus 0.4em\relax
  IEEE, 2016, pp. 3636--3641.

\bibitem{mao2017successive}
Y.~Mao, D.~Dueri, M.~Szmuk, and B.~A{\c{c}}{\i}kme{\c{s}}e, ``Successive
  convexification of non-convex optimal control problems with state
  constraints,'' \emph{IFAC-PapersOnLine}, vol.~50, no.~1, pp. 4063--4069,
  2017.

\bibitem{mao2018successive}
Y.~Mao, M.~Szmuk, X.~Xu, and B.~A{\c{c}}ikmese, ``Successive convexification: A
  superlinearly convergent algorithm for non-convex optimal control problems,''
  \emph{arXiv preprint arXiv:1804.06539[math.OC]}, 2018.

\bibitem{patrinos2013accelerated}
P.~Patrinos and A.~Bemporad, ``An accelerated dual gradient-projection
  algorithm for embedded linear model predictive control,'' \emph{IEEE Trans.
  Automat. Control}, vol.~59, no.~1, pp. 18--33, 2013.

\bibitem{giselsson2014improved}
P.~Giselsson, ``Improved fast dual gradient methods for embedded model
  predictive control,'' \emph{IFAC Proc. Vol.}, vol.~47, no.~3, pp. 2303--2309,
  2014.

\bibitem{kufoalor2014embedded}
D.~K.~M. Kufoalor, S.~Richter, L.~Imsland, T.~A. Johansen, M.~Morari, and G.~O.
  Eikrem, ``Embedded model predictive control on a {PLC} using a primal-dual
  first-order method for a subsea separation process,'' in \emph{Proc.
  Mediterranean Conf. Control Autom.}\hskip 1em plus 0.5em minus 0.4em\relax
  IEEE, 2014, pp. 368--373.

\bibitem{o2013splitting}
B.~O'Donoghue, G.~Stathopoulos, and S.~Boyd, ``A splitting method for optimal
  control,'' \emph{IEEE Trans. Control Syst. Technol.}, vol.~21, no.~6, pp.
  2432--2442, 2013.

\bibitem{sokoler2014input}
L.~E. Sokoler, G.~Frison, M.~S. Andersen, and J.~B. J{\o}rgensen,
  ``Input-constrained model predictive control via the alternating direction
  method of multipliers,'' in \emph{Proc. Eur. Control Conf.}\hskip 1em plus
  0.5em minus 0.4em\relax IEEE, 2014, pp. 115--120.

\bibitem{dang2015embedded}
T.~V. Dang, K.~V. Ling, and J.~M. Maciejowski, ``Embedded {ADMM}-based {QP}
  solver for {MPC} with polytopic constraints,'' in \emph{Proc. Eur. Control
  Conf.}\hskip 1em plus 0.5em minus 0.4em\relax IEEE, 2015, pp. 3446--3451.

\bibitem{wang2009fast}
Y.~Wang and S.~Boyd, ``Fast model predictive control using online
  optimization,'' \emph{IEEE Trans. Control Syst. Technol.}, vol.~18, no.~2,
  pp. 267--278, 2009.

\bibitem{kogel2011fast}
M.~K{\"o}gel and R.~Findeisen, ``Fast predictive control of linear systems
  combining {N}esterov's gradient method and the method of multipliers,'' in
  \emph{Proc. IEEE Conf. Decision Control and Eur. Control Conf.}\hskip 1em
  plus 0.5em minus 0.4em\relax IEEE, 2011, pp. 501--506.

\bibitem{patrinos2011global}
P.~Patrinos, P.~Sopasakis, and H.~Sarimveis, ``A global piecewise smooth
  {N}ewton method for fast large-scale model predictive control,''
  \emph{Automatica}, vol.~47, no.~9, pp. 2016--2022, 2011.

\bibitem{blanchard2017sor}
H.~A. Blanchard and A.~A. Adegbege, ``An {SOR}-like method for fast model
  predictive control,'' \emph{IFAC-PapersOnLine}, vol.~50, no.~1, pp.
  14\,418--14\,423, 2017.

\bibitem{chambolle2016ergodic}
A.~Chambolle and T.~Pock, ``On the ergodic convergence rates of a first-order
  primal--dual algorithm,'' \emph{Math. Program.}, vol. 159, no. 1-2, pp.
  253--287, 2016.

\bibitem{seidman2018chebyshev}
J.~H. Seidman, M.~Fazlyab, G.~J. Pappas, and V.~M. Preciado, ``A
  chebyshev-accelerated primal-dual method for distributed optimization,'' in
  \emph{Proc. IEEE Conf. Decision Control}.\hskip 1em plus 0.5em minus
  0.4em\relax IEEE, 2018, pp. 1775--1781.

\bibitem{nesterov2018lectures}
Y.~Nesterov, \emph{Lectures on convex optimization}.\hskip 1em plus 0.5em minus
  0.4em\relax Springer, 2010, vol. 137.

\bibitem{kouzoupis2018recent}
D.~Kouzoupis, G.~Frison, A.~Zanelli, and M.~Diehl, ``Recent advances in
  quadratic programming algorithms for nonlinear model predictive control,''
  \emph{Vietnam J. Math.}, vol.~46, no.~4, pp. 863--882, 2018.

\bibitem{wang2010control}
J.~Wang and N.~Elia, ``Control approach to distributed optimization,'' in
  \emph{Proc. Allerton Conf. Commun. Control Comput.}\hskip 1em plus 0.5em
  minus 0.4em\relax IEEE, 2010, pp. 557--561.

\bibitem{yu2020mass}
Y.~Yu, B.~A{\c{c}}{\i}kme{\c{s}}e, and M.~Mesbahi, ``Mass--spring--damper
  networks for distributed optimization in non-{E}uclidean spaces,''
  \emph{Automatica}, vol. 112, p. 108703, 2020.

\bibitem{yu2020rlc}
Y.~Yu and B.~A{\c{c}}{\i}kme{\c{s}}e, ``{RLC} circuits-based distributed mirror
  descent method,'' \emph{IEEE Control Syst. Lett.}, vol.~4, no.~3, pp.
  548--553, 2020.

\bibitem{ouyang2019lower}
Y.~Ouyang and Y.~Xu, ``Lower complexity bounds of first-order methods for
  convex-concave bilinear saddle-point problems,'' \emph{Math. Program.}, pp.
  1--35, 2019.

\bibitem{lacoste2012simpler}
S.~Lacoste-Julien, M.~Schmidt, and F.~Bach, ``A simpler approach to obtaining
  an {O}(1/t) convergence rate for the projected stochastic subgradient
  method,'' \emph{arXiv preprint arXiv:1212.2002[cs.LG]}, 2012.

\bibitem{xu2017accelerated}
Y.~Xu, ``Accelerated first-order primal-dual proximal methods for linearly
  constrained composite convex programming,'' \emph{SIAM J. Optim.}, vol.~27,
  no.~3, pp. 1459--1484, 2017.

\bibitem{zagaris2018model}
C.~Zagaris, H.~Park, J.~Virgili-Llop, R.~Zappulla, M.~Romano, and
  I.~Kolmanovsky, ``Model predictive control of spacecraft relative motion with
  convexified keep-out-zone constraints,'' \emph{J. Guid. Control Dyn.},
  vol.~41, no.~9, pp. 2054--2062, 2018.

\bibitem{domahidi2013ecos}
A.~Domahidi, E.~Chu, and S.~Boyd, ``{ECOS}: An {SOCP} solver for embedded
  systems,'' in \emph{Proc. Eur. Control Conf.}\hskip 1em plus 0.5em minus
  0.4em\relax IEEE, 2013, pp. 3071--3076.

\bibitem{dunning2017jump}
I.~Dunning, J.~Huchette, and M.~Lubin, ``{JuMP}: A modeling language for
  mathematical optimization,'' \emph{SIAM Rev.}, vol.~59, no.~2, pp. 295--320,
  2017.

\end{thebibliography}

\end{document}